\documentclass{birkjour}
 \newtheorem{theorem}{Theorem}[section]
 
 \newtheorem{lemma}[theorem]{Lemma}
 
 \theoremstyle{definition}
 
 \theoremstyle{remark}
 \newtheorem{remark}[theorem]{Remark}
 
 \numberwithin{equation}{section}

 \usepackage{hyperref}
\usepackage{graphicx}

\def\dis{\displaystyle}

\def\dx{\hbox{dx}}
\def\ds{\hbox{ds}}
\def\dy{\hbox{dy}}
\def\eps{\varepsilon}

\def\O{\omega}
\def\K{\mathcal K}

\begin{document}

%-------------------------------------------------------------------------
% editorial commands: to be inserted by the editorial office
%
%\firstpage{1} \volume{228} \Copyrightyear{2004} \DOI{003-0001}
%\seriesextraline{This is the Concrete Title of this Book\br H.E. R and S.T.C. W, Eds.}
%\firstpage{1}

\title[]
 {A non-iterative reconstruction method for an inverse problem
modeled by a Stokes-Brinkmann equations}

%----------Author 1
\author[Mourad]{Mourad Hrizi}

\address{Monastir University, Department of Mathematics, Faculty of Sciences
Avenue de l'Environnement $5000,$ Monastir, Tunisia}
\email{mourad-hrizi@hotmail.fr}

%\thanks{This work was completed with the support of our
%----------Author 2
\author{Rakia Malek}
\address{Monastir University, Department of Mathematics, Faculty of Sciences
Avenue de l'Environnement $5000,$ Monastir, Tunisia}
\email{rakia\_malek@hotmail.fr}

\author{Maatoug Hassine}
\address{Monastir University, Department of Mathematics, Faculty of Sciences
Avenue de l'Environnement $5000,$ Monastir, Tunisia}
\email{maatoug.hassine@enit.rnu.tn}


%\date{January 1, 2004}
%\dedicatory{}
%%% ----------------------------------------------------------------------

\begin{abstract}
This article is concerned with the reconstruction of obstacle $\O$
immersed in a fluid flowing in a bounded domain $\Omega$ in the two
dimensional case. We assume that the fluid motion is governed by the
Stokes-Brinkmann equations. We make an internal measurement and then
have a least-square approach to locate the obstacle. The idea is to
rewrite the reconstruction problem as a topology optimization
problem. The existence and the stability of the optimization problem
are demonstrated. We use here the concept of the topological
gradient in order to determine the obstacle and it's rough location.
The topological gradient is computed using a straightforward way
based on a penalization technique without the truncation method used
in the literature. The unknown obstacle is reconstructed using a
level-set curve of the topological gradient. Finally, we make some
numerical examples exploring the efficiency of the method.
\end{abstract}

%%% ----------------------------------------------------------------------
\maketitle
%%% ----------------------------------------------------------------------
\tableofcontents

%%%%%%%%%%%%%%%%%%%%%%%%%%%%%%%%%%%%%%%%%%%%%%%%%%%%%%%%%%%%%%%%%%%%%%%%%%
\section{Introduction}
%%%%%%%%%%%%%%%%%%%%%%%%%%%%%%%%%%%%%%%%%%%%%%%%%%%%%%%%%%%%%%%%%%%%%%%%%%
This paper is concerned with an inverse problem related to the
Stokes-Brinkmann equations. It consists of reconstructing an
obstacle immersed in a porous media $\Omega\subset\mathbb{R}^2$ with
the help of collecting measurements of the velocity of the fluid
motion. Such an inverse problem has several applications, for
instance in modeling of liquids or gas through the ground
\cite{krotkiewski2011importance,popov2009multiscale} and
microfluidics \cite{koster2007numerical}.

The number of publications on inverse problems for Stokes-Brinkmann
equations are relatively small when compared to Stokes equations,
see for instance
\cite{abda2009topological,caubet2012localization,caubet2015detection,beretta2017size,heck2007reconstruction,alves2004determination}
and reference therein. The works which are related to ours are
presented by Lechleiter and Rienm\"{u}ller
\cite{lechleiter2013factorization} and Yan \emph{et al.}
\cite{yan2019shape}. In the first reference, they identified the
shape of a penetrable inclusion from boundary measurements using the
factorization method. In the work of Yan \emph{et al,} they solved
the considered inverse problem and proposed a method relies on the
minimization of a tracking cost functional using the shape gradient
method. They derived the shape gradient for the tracking functional
based on the continuous adjoint method and the function space
parametrization technique.

In our paper, to solve this inverse problem numerically, we propose
an alternative method based on the topological sensitivity analysis.
The general idea of the proposed method consists in rewriting the
inverse problem as a topology optimization problem, where the
obstacle is the unknown variable. The topology optimization problem
consists in minimizing the so-called least squares functional with
the total variation regularization. This cost functional is
minimized with respect to a small topological perturbation of the
obstacle by using the concept of topological sensitivity. The main
advantage of this detection method is that, it provides fast and
accurate results for detection.

The topological sensitivity analysis consists of studying the
variation of a given cost functional with respect to the presence of
a small domain perturbation, such as  the insertion of inclusions,
cavities, cracks or source-terms. Let us briefly discuss the history
of this method. Its main idea was originally introduced by
Schumacher \cite{schumacher1996topologieoptimierung} in the context
of compliance minimization in linear elasticity. In the same context
Sokolowski $\&$ Zochowski \cite{sokolowski1999topological}, who
studied the effect of an extract infinitesimal part of the material
in structural mechanics. Then in \cite{masmoudi2002topological}
Masmoudi worked out a topological sensitivity analysis framework
based on a generalization of the adjoint method and on the use of a
truncation technique. By using this framework the topological
sensitivity is obtained for several equations
\cite{garreau2001topological,masmoudi2005topological,pommier2004topological,samet2003topological}.
For other works on the topological sensitivity concept, we refer to
the book by Novotny $\&$ Sokolowski \cite{novotny2012topological}.

In order to introduce this concept, let us consider a bounded domain
$\Omega\subset\mathbb{R}^2$ and a cost function
$j(\Omega)=\mathcal{J}(\psi_\Omega)$ to be minimized, where
$\psi_\Omega$
 is the solution to a given partial
differential equation defined in $\Omega.$ For $\varepsilon>0,$ let
$\Omega\backslash\overline{\mathcal{S}_{z,\varepsilon}}$ be the
perturbed domain obtained by removing a small topological
perturbation $\mathcal{S}_{z,\varepsilon}=z+\varepsilon\mathcal{S}$
from the reference (unperturbed) domain $\Omega,$ where $z\in\Omega$
and $\mathcal{S}\subset\mathbb{R}^2$ is a given fixed and bounded
domain containing the origin. The topological sensitivity analysis
leads to an asymptotic expansion of the function $j$ of the form
$$
j(\Omega\backslash\overline{\mathcal{S}_{z,\varepsilon}})=j(\Omega)+f(\varepsilon)\delta
j(z)+o(f(\varepsilon)),
$$
where $f(\varepsilon)$ is a positive function depending upon the
size $\varepsilon$ of the topological perturbation such that
$f\rightarrow0,$ when $\varepsilon\rightarrow0.$ The function
$z\mapsto\delta j(z)$ is called the ``topological gradient" or
``topological sensitivity" of $j$ at $z.$ Mathematically, we can
express it as
$$
\delta
j(z):=\lim_{\varepsilon\rightarrow0}\frac{j(\Omega\backslash\overline{\mathcal{S}_{z,\varepsilon}})-j(\Omega)}{f(\varepsilon)}.
$$
Hence, if we want to minimize the cost function $j$, the best
location to insert a small perturbation in $\Omega$ is where $\delta
j$ is most negative. In fact if $\delta j(z)<0$, we have
$j(\Omega\backslash\overline{\mathcal{S}_{z,\varepsilon}})\leq
j(\Omega)$ for small $\varepsilon.$ Topological sensitivity analysis
for the Stokes equations has been studied in the past by Guillaume
and Idris \cite{guillaume2004topological}, they used the Masmoudi's
approach which is truncation technique. For the quasi-Stokes
\cite{hassine2004topological} problems the topological sensitivity
has been treated again with the truncation technique. For
Navier-Stokes equations, we refer the reader to the work
\cite{amstutz2005topological}, where the topological sensitivity is
computed with the help of an alternative to the truncation based on
the comparison between the perturbed and the initial problems both
formulated in the perforated domain.

In this paper, we have derived a topological asymptotic expansion of
the cost functional by using a penalization technique. This approach
allowed us to perform the topological asymptotic without using the
truncation method presented in the previous works. From the obtained
theoretical results, we propose a fast and accurate detection
algorithm for recovering the shape and the location of an obstacle.
The efficiency and accuracy of the proposed algorithm are
illustrated by some numerical examples. Particularly, we test the
influence of some parameters in our procedure such as the shape,
location and the size of the obstacles.

The rest of this paper is organized as follows. In Section
\ref{section1}, we introduce the notation for function spaces and we
present the forward and inverse problem. Section \ref{section2}
proves the unique existence and the stability of the considered
optimization problem. While in Section \ref{section4}, we derive the
asymptotic expansion of the proposed cost functional. In Section
\ref{numerique}, some numerical experiments are presented in order
to show the effectiveness of the proposed method. Finally, the paper
ends with some concluding remarks in Section \ref{conclusion}.

%%%%%%%%%%%%%%%%%%%%%%%%%%%%%%%%%%%%%%%%%%%%%%%%%%%%%%%%%%%%%%%%%%%%%%%%%%%
\section{The problem setting}\label{section1}
%%%%%%%%%%%%%%%%%%%%%%%%%%%%%%%%%%%%%%%%%%%%%%%%%%%%%%%%%%%%%%%%%%%%%%%%%%%

%%%%%%%%%%%%%%%%%%%%%%%%%%%%%%%%%%%%%%%%%%%%%%%%%%%%%%%%%%%%%%%%%%%%%%%%%%%
\subsection{Notation}
%%%%%%%%%%%%%%%%%%%%%%%%%%%%%%%%%%%%%%%%%%%%%%%%%%%%%%%%%%%%%%%%%%%%%%%%%%%
Let us introduce some notation which will be useful in what follows.
For an open and bounded domain $\Omega\subset\mathbb{R}^2,$ we
denote by $L^q(\Omega):=[L^q(\Omega)]^2$ and
$H^s(\Omega):=[H^s(\Omega)]^2$ the usual Lebesgue and Sobolev
spaces. We define an inner product for matrices by
$M:N=\sum_{i,j=1}^2M_{ij}N_{i,j}$ for $M,N\in\mathbb{R}^{2\times2};$
the associated norm is $|M|=\sqrt{M:M}.$ The corresponding inner
product on $L^2(\Omega)^{2\times2}$ is
$$
\langle M,N\rangle_{L^2(\Omega)^{2\times2}}=\int_\Omega M(x):N(x)\
\dx\ \ \hbox{for}\ M,N\in L^2(\Omega)^{2\times2}.
$$
In a Banach space $\mathcal{Y}$, we denote the weak convergence of a
sequence $\{\zeta_n\}_n$ to $\zeta$ by
$$
\zeta_n\rightharpoonup\zeta \ \ \hbox{in}\ \ \mathcal{Y}\ \
\hbox{as}\ n\rightarrow\infty.
$$
Finally, for the sake of completeness we briefly introduce the space
of functions with bounded total variation. Standard properties of
bounded variation functions can be found in
\cite{ambrosio2000functions,attouch2006variational}. A function $u$
belonging to $L^1(\Omega)$ is said to be of bounded total variation
if $$ TV(u)=\int_\Omega D u(x)\dx:=\sup\Big\{\int_\Omega u\
\hbox{div}\ \varphi\ \dx\Big|\
\varphi\in\mathcal{C}^1_c(\Omega,\mathbb{R}^d),\
\|\varphi\|_{L^\infty(\Omega)}\leq1 \Big\}<\infty.
$$
Here $\mathcal{C}^1_c(\Omega,\mathbb{R}^d)$ is the space of
continuously differentiable functions with compact support in
$\Omega$ and $\|.\|_{L^\infty(\Omega)}$ is the essential supremum
norm. The space of all functions in $L^1(\Omega)$ with bounded total
variation is denoted by
$$
BV(\Omega)=\Big\{u\in L^1(\Omega)\Big|\ \int_\Omega D
u(x)\dx<\infty\Big\}.
$$

%%%%%%%%%%%%%%%%%%%%%%%%%%%%%%%%%%%%%%%%%%%%%%%%%%%%%%%%%%%%%%%%%%%%%%%%%%%%%%%%%%%%%
\subsection{The studied problem}
%%%%%%%%%%%%%%%%%%%%%%%%%%%%%%%%%%%%%%%%%%%%%%%%%%%%%%%%%%%%%%%%%%%%%%%%%%%%%%%%%%%%%
Let $\Omega$ be a bounded Lipschitz open set of $\mathbb{R}^2$
containing a Newtonian and incompressible fluid with coefficient of
kinematic viscosity $\nu>0$ and has an inverse permeability
$\alpha>0.$ Let $\omega$ be a bounded Lipschitz domain included in
$\Omega.$ The Brinkmann system describing the motion of the fluid in
$\Omega$ in the presence of the obstacle $\O$ is given by (see, for
example \cite{borrvall2003topology})
\begin{equation}\label{11}
\left\{
\begin{array}{c}
\begin{array}{r l l l}
-\nu\Delta \psi+\alpha \psi+\nabla p & = 0& \mbox{in } & { \Omega\backslash \overline{{\O}},} \\%[0.12cm]
 \mbox{div } \psi & =  0 & \mbox{in } &{ \Omega\backslash \overline{{\O}},}\\%[0.12cm]
  \psi& =0 &  \mbox {on }&{  \Gamma,}\\%[0.12cm]
  \sigma(\psi,p)\textbf{n}& =g &  \mbox {on }&{  \Sigma,}\\%[0.12cm]
\psi & =  0 & \mbox {on }&{\partial\O},
\end{array}
\end{array}
 \right.
\end{equation}
where $g\in H^{-1/2}(\Sigma)$ is a given function, $\psi$ represents
the velocity of the fluid and $p$ the pressure and $\sigma$
represents the stress tensor defined by
\begin{align*}
\sigma(\psi,p)=-p\mathrm{I}+2\nu e(\psi),
\end{align*}
with $\mathrm{I}$ is the $2\times2$ identity matrix and $e(\psi)$ is
the linear strain tensor defined as
\begin{align*}
e(\psi)=\frac{1}{2}\Big(\nabla \psi+\ ^t\nabla \psi\Big).
\end{align*}
Here $\textbf{n}$ denotes the outward normal to the boundary
$\partial\Omega=\Sigma\cup\Gamma$ where $\Sigma$ and $\Gamma$ have
both a nonnegative Lebesgue measure and
$\Sigma\cap\Gamma=\emptyset.$
%The existence and uniqueness of the solution of (\ref{11}) such a
%problem is classical. We refer for instance to the book of ???
%\cite{}.

The aim of this work is to reconstruct the obstacle $\O.$ In order
to reconstruct the location and the shape of the obstacle, we make a
measurement $\psi^d\in H^1(\Omega)$. Thus we consider the following
geometric inverse problem:
\begin{align}\label{inverse}
Determine\ the\ obstacle\ \omega\subset\Omega\ from\ the\
measurement\ \psi^d.
\end{align}
The data $\psi^d$ for which this inverse problem has a solution
$\psi$ are said to be \emph{compatible}.

To solve numerically this geometric inverse problem, we introduce
the following optimization problem:
\begin{equation}\label{minimizationpb}
\left\{
  \begin{array}{ll}
    \operatorname*{Minimize}\
\K(\O,\psi):=\mathrm{J}(\O)=\dis\int_{\Omega\backslash\overline{\O}}\Big|\psi-\psi^d\Big|^2\dx+\rho\mathcal{P}(\O,\Omega), \\
    \hbox{subjet to}\ \ \O\in\mathcal{U}_{ad}\ \ \hbox{and} \ \ \psi\ \ \hbox{is the solution
to}\ \ (\ref{11})
  \end{array}
\right.
\end{equation}
where $\rho$ is a regularization parameter and
\begin{equation*}
\mathcal{U}_{ad}=\Big\{\O\subset\Omega: \O \ \hbox{is a subdomain
in}\ \Omega\ \hbox{such that}\ \mathcal{P}(\O,\Omega)<+\infty\Big\},
\end{equation*}
with $\mathcal{P}(\O,\Omega)$ denotes the relative perimeter of $\O$
in $\Omega$ is defined by
\begin{align*}
\mathcal{P}(\O,\Omega):=TV(\chi(\O))=\sup\Big\{\int_\O\hbox{div}\
\varphi\ \dx:\ \varphi\in\mathcal{C}^1_c(\Omega,\mathbb{R}^2),\
\|\varphi\|_{L^\infty(\Omega)}\leq1\Big\}.
\end{align*}
Here $\chi(\O)$ is the characteristic function of $\O.$

\begin{remark}If $\mathcal{P}(\O,\Omega)<+\infty,$ we say that $\O$ has
finite perimeter in $\Omega.$ In this case the relative perimeter
$\mathcal{P}(\O,\Omega)$ of $\O$ coincides with the total variation
of the distributional gradient of the characteristic function of
$\O:$
$$\mathcal{P}(\O,\Omega)=|D\chi(\O)|(\Omega).$$
\end{remark}

%%%%%%%%%%%%%%%%%%%%%%%%%%%%%%%%%%%%%%%%%%%%%%%%%%%%%%%%%%%
\section{Preliminary results}\label{section2}
%%%%%%%%%%%%%%%%%%%%%%%%%%%%%%%%%%%%%%%%%%%%%%%%%%%%%%%%%%%
In this section, we shall establish the unique existence of the
solution as well as the stability of the minimization problem
(\ref{minimizationpb}).

Before establishing uniqueness and stability of the minimization
problem (\ref{minimizationpb}), we need the following lemma given in
\cite{girault2012finite,lechleiter2013factorization}.

\begin{lemma}\label{lemma1}
The boundary value problem (\ref{11}) admits a unique solution
$(\psi(\O),p(\O))$ and there exists a constant $c>0$ such that
$$
\Big\|\psi(\O)\Big\|_{H^1(\Omega)}\leq
c\Big\|g\Big\|_{H^{-1/2}(\Sigma)}.
$$
\end{lemma}

\begin{remark}
In the above inequality, the solution $\psi(\O)$ is extended by zero
inside the domain $\O,$ still denoted by $\psi(\O)$.
\end{remark}

The penalization of the cost function ($L^2$-norm) by the relative
perimeter is relevant for the existence and uniqueness of the
optimal solution of (\ref{minimizationpb}), which will be proved in
the following theorem.
\begin{theorem}\label{theorem2.2}
For any $\psi^{d}\in H^1(\Omega),$ there exists a unique minimizer
$\O^*\in\mathcal{U}_{ad}$ to the minimization problem
(\ref{minimizationpb}).
\end{theorem}
\begin{proof}
Since $\mathrm{J}(\O)$ is non-negative, we know that
$\displaystyle\inf_{\O\in\mathcal{U}_{ad}}\mathrm{J}(\O)$ is finite.
Therefore, there exists a minimizing sequence
$\{\O_n\}_n\subset\mathcal{U}_{ad}$ such that
$$
\lim_{n\rightarrow\infty}\mathrm{J}(\O_n)=\inf_{\O\in\mathcal{U}_{ad}}\mathrm{J}(\O).
$$
From the definition of the admissible set $\mathcal{U}_{ad}$, we
have $\mathcal{P}(\O_n,\Omega)<+\infty$  then $\{\chi(\O_n)\}_n$ is
bounded in $BV(\Omega).$ Thus, $\{\chi(\O_n)\}_n$ is relatively
compact in $L^1(\Omega).$ Therefore, there exists
$\O^*\in\mathcal{U}_{ad}$ and a subsequence of $\{\chi(\O_n)\}_n,$
still denoted by $\{\chi(\O_n)\}_n,$ such that
$$
\chi(\O_n)\rightarrow\chi(\O^*) \ \ \hbox{in}\ L^1(\Omega)\
\hbox{as}\ n\rightarrow\infty.
$$
%(i.e $ \chi(\O_n)\rightharpoonup\chi(\O^*) \ \ \hbox{in}\
%\mathrm{BV}(\Omega)\ \hbox{as}\ n\rightarrow\infty$).
Now we prove that $\O^*$ is indeed the unique minimizer to the
problem (\ref{minimizationpb}).

Since each $\O_n$ corresponds with a solution $\psi(\O_n)$ to
(\ref{11}) with $\O=\O_n,$ it follows immediately from Lemma
\ref{lemma1} that the sequence $\{\psi(\O_n)\}_n$ is also bounded in
$H^1(\Omega)$. This indicates the existence of some $\psi^*\in
H^1(\Omega)$ and a subsequence of $\{\psi(\O_n)\}_n,$ again still
denoted by $\{\psi(\O_n)\}_n,$ such that
\begin{align}\label{convergancefaible}
 \psi(\O_n)\rightharpoonup\psi^* \ \hbox{in}\ H^1(\Omega)\ \hbox{as}\ n\rightarrow\infty.
\end{align}
We claim $\psi^*=\psi(\O^*).$ Actually, using Green's formula on
(\ref{11}), we have
$$
\int_{\Omega\backslash\overline{\O}}\Big(\nu\nabla\psi:\nabla\vartheta+\alpha\psi.\vartheta\Big)\
\dx=\int_\Sigma g.\vartheta\ \ds, \ \ \hbox{for all} \
\vartheta\in\mathcal{V}(\O),
$$
where the functional space $\mathcal{V}(\O)$ is defined by
$$\mathcal{V}(\O)=\Big\{v\in H^1(\Omega\backslash\overline{\O});\
\hbox{div} v=0 \ \hbox{in}\ \Omega\ \hbox{and}\ v=0\ \hbox{on}\
\Gamma\cup\partial\O\Big\}.$$ By taking $\O=\O_n$ and
$\psi=\psi(\O_n)$ we have
\begin{align}\label{vrr}
\int_{\Omega}\chi(\Omega\backslash\overline{\O_n})\Big(\nu\nabla\psi(\O_n):\nabla\vartheta+\alpha\psi(\O_n).\vartheta\Big)\
\dx=\int_{\Sigma}g.\vartheta\ \ds, \ \ \hbox{for all} \
\vartheta\in\mathcal{V}(\O_n).
\end{align}
Since (\ref{convergancefaible}) implies
$$
 \nabla\psi(\O_n)\rightharpoonup\nabla\psi^* \ \hbox{in}\ H^1(\Omega)\ \hbox{as}\
n\rightarrow\infty,
$$
we pass $n\rightarrow\infty$ in (\ref{vrr}) to obtain
\begin{align*}
\int_{\Omega}\chi(\Omega\backslash\overline{\O^*})\Big(\nu\nabla\psi^*:\nabla\vartheta+\alpha\psi^*.\vartheta\Big)\
\dx=\int_{\Sigma}g.\vartheta\ \ds, \ \ \hbox{for all} \
\vartheta\in\mathcal{V}(\O^*).
\end{align*}
Then it follows from the definition of weak solution and Lemma
\ref{lemma1} that $\psi^*$ coincides with the unique solution to
(\ref{11}) with $\O=\O^*,$ that is, $\psi^*=\psi(\O^*).$

Finally, using $\chi(\O_n)\rightarrow\chi(\O^*) \ \ \hbox{in}\
L^1(\Omega)$ and (\ref{convergancefaible}), we employ the lower
semi-continuity of the $L^2$-norm and the lower semi-continuity of
the perimeter to conclude
\begin{align*}
\mathrm{J}(\O^*)&=\int_{\Omega\backslash\overline{\O^*}}\Big|\psi^*-\psi^d\Big|^2\dx+\rho\mathcal{P}(\O^*,\Omega)\\
&\leq\lim_{n\rightarrow\infty}\inf\int_{\Omega\backslash\overline{\O_n}}\Big|\psi(\O_n)-\psi^d\Big|^2\dx+\rho\lim_{n\rightarrow\infty}\inf\mathcal{P}(\O_n,\Omega)\\
&\leq\lim_{n\rightarrow\infty}\inf\mathrm{J}(\O_n)=\inf_{\O\in\mathcal{U}_{ad}}\mathrm{J}(\O).
\end{align*}
\end{proof}

Next, we justify the stability of (\ref{minimizationpb}), namely,
the minimization problem (\ref{minimizationpb}) that is indeed a
stabilization for problem (\ref{11}) with respect to the observation
data $\psi^d.$

\begin{theorem}\label{theorem2.3}
Let $\{\psi^d_n\}_n\subset H^1(\Omega)$ be a sequence such that
\begin{align}\label{es4}
\psi^d_n\rightharpoonup \psi^d\ \ \hbox{in} \ H^1(\Omega)\
\hbox{as}\ n\rightarrow\infty,
\end{align}
and $\{\O_n\}_n$ be a sequence of minimizer of problems
$$
\operatorname*{Minimize}_{\O\in\mathcal{U}_{ad}}\ \mathrm{J}_n(\O)\
\ \hbox{with}\ \
\mathrm{J}_n(\O):=\int_{\Omega\backslash\overline{\O}}\Big|\psi-\psi^d_n\Big|^2\dx+\rho\mathcal{P}(\O,\Omega),\
\ n=1,2,....
$$
Then $\{\O_n\}_n$ converges weakly in $H^1(\Omega)$ to the minimizer
of (\ref{minimizationpb}).
\end{theorem}
\begin{proof}
The unique existence of each $\O_n$ is guaranteed by Theorem
\ref{theorem2.2}. By definition, we get:
$$
\mathrm{J}_n(\O_n)\leq\mathrm{J}_n(\O),\ \ \forall
\O\in\mathcal{U}_{ad},
$$
which implies the boundedness of $\chi(\O_n)$ in $BV(\Omega).$ Thus
$\chi(\O_n)$ are relatively compact in $\mathrm{L}^1(\Omega).$
Hence, there exists $\O^*\in\mathcal{U}_{ad}$ and a subsequence of
$\{\chi(\O_n)\}_n,$ still denoted by $\{\chi(\O_n)\}_n,$ such that
$$
\chi(\O_n)\rightarrow \chi(\O^*)\ \ \hbox{in}\ L^1(\Omega)\
\hbox{as}\ n\rightarrow\infty.
$$
Now it suffices to show that $\O^*$ is indeed the unique minimizer
of (\ref{minimizationpb}). Actually, repeating the same argument as
that in the proof of Theorem \ref{theorem2.2}, we can derive
\begin{align}\label{es44}
\psi(\O_n)\rightharpoonup\psi(\O^*)\ \ \hbox{in}\ \ H^1(\Omega)\
\hbox{as}\ n\rightarrow\infty.
\end{align}
up to taking a further subsequence. Gathering (\ref{es4}) and
(\ref{es44}), we obtain
\begin{align*}
\psi(\O_n)-\psi^d_n\rightharpoonup\psi(\O^*)-\psi^d\ \ \hbox{in}\ \
H^1(\Omega)\ \hbox{as}\ n\rightarrow\infty.
\end{align*}
Consequently, for any $\O\in\mathcal{U}_{ad},$ again we take
advantage of the the lower semi-continuity of the $L^2$-norm and the
lower semi-continuity of the perimeter to deduce
\begin{align*}
\mathrm{J}(\O^*)&=\int_{\Omega\backslash\overline{\O^*}}\Big|\psi^*-\psi^d\Big|^2\dx+\rho\mathcal{P}(\O^*,\Omega)\\
&\leq\lim_{n\rightarrow\infty}\inf\int_{\Omega\backslash\overline{\O_n}}\Big|\psi(\O_n)-\psi^d_n\Big|^2\dx+\rho\lim_{n\rightarrow\infty}\inf\mathcal{P}(\O_n,\Omega)\\
&\leq\lim_{n\rightarrow\infty}\inf\Big[\int_{\Omega\backslash\overline{\O_n}}\Big|\psi(\O_n)-\psi^d_n\Big|^2\dx+\rho\mathcal{P}(\O_n,\Omega)\Big]\\
&\leq\lim_{n\rightarrow\infty}\Big[\int_{\Omega\backslash\overline{\O}}\Big|\psi(\O)-\psi^d_n\Big|^2\dx+\rho\mathcal{P}(\O,\Omega)\Big]\\
&=\int_{\Omega\backslash\overline{\O}}\Big|\psi(\O)-\psi^d\Big|^2\dx+\rho\mathcal{P}(\O,\Omega)=\mathrm{J}(\O),\
\ \forall\O\in\mathcal{U}_{ad},
\end{align*}
which verifies that $\O^*$ is the minimizer of
(\ref{minimizationpb}).
\end{proof}

To solve the minimization problem (\ref{minimizationpb}), we
introduce the topological sensitivity analysis method.

%%%%%%%%%%%%%%%%%%%%%%%%%%%%%%%%%%%%%%%%%%%%%%%%%%%%%%%%%%%%%%%%%%%%%%%%%%%%%%%%%%%%%%%%%%%%%%%%%
\section{Topological sensitivity analysis}\label{section4}
%%%%%%%%%%%%%%%%%%%%%%%%%%%%%%%%%%%%%%%%%%%%%%%%%%%%%%%%%%%%%%%%%%%%%%%%%%%%%%%%%%%%%%%%%%%%%%%%%
In this section, we derive the asymptotic expansion of the cost
functional $\K$ with respect to the insertion of a small obstacle
$\O_{z,\eps}\subset\subset\Omega$ that is centered at $z\in\Omega$
and has the form $\O_{z,\eps}=z+\eps\omega$ where $\eps$ is a small
parameter and $\omega$ is a given bounded domain.

In the presence of the perturbed obstacle $\O_{z,\eps},$ the
velocity $\psi_\eps$ and the pressure $p_\eps$ solve the following
Brinkmann problem:
\begin{equation}\label{2}
\left\{
\begin{array}{c}
\begin{array}{r l l l}
-\nu\Delta \psi_\eps+\alpha \psi_\eps+\nabla p_\eps & = 0& \mbox{in } & { \Omega\backslash \overline{{\O_{z,\eps}}},} \\%[0.12cm]
 \mbox{div } \psi_\eps & =  0 & \mbox{in } &{ \Omega\backslash \overline{{\O_{z,\eps}}},}\\%[0.12cm]
  \psi_\eps& =0 &  \mbox {on }&{  \Gamma,}\\%[0.12cm]
  \sigma(\psi_\eps,p_\eps)\textbf{n}& =g &  \mbox {on }&{  \Sigma,}\\%[0.12cm]
\psi_\eps & =  0 & \mbox {on }&{\partial\O_{z,\eps}}.
\end{array}
\end{array}
 \right.
\end{equation}
Using the penalization technique used in the finite element method
for the implementation of a Dirichlet condition, we can rewrite
problem (\ref{2}) as
%we conclude that the perturbed problem (\ref{2}) is equivalent to
%the following penalized problem:
\begin{equation}\label{22}
\left\{
\begin{array}{c}
\begin{array}{r l l l}
-\nu\Delta \psi_\eps+\alpha \psi_\eps+\delta c_\eps\psi_\eps+\nabla p_\eps & = 0& \mbox{in } & { \Omega,} \\%[0.12cm]
 \mbox{div } \psi_\eps & =  0 & \mbox{in } &{ \Omega,}\\%[0.12cm]
  \psi_\eps& =0 &  \mbox {on }&{  \Gamma,}\\%[0.12cm]
  \sigma(\psi_\eps,p_\eps)\textbf{n}& =g &  \mbox {on }&{  \Sigma,}\\%[0.12cm]
\end{array}
\end{array}
 \right.
\end{equation}
where $\delta c_\eps$ is a piecewise constant function defined by
$$
\delta c_\eps(x)=\left\{
                   \begin{array}{ll}
                     k & \hbox{if}\ x\in\O_{z,\eps}, \\
                     0 & \hbox{if}\
x\in\Omega\backslash\overline{\O_{z,\eps}},
                   \end{array}
                 \right.
$$
where $k$ is large enough. The weak form associated with (\ref{22})
reads:
\begin{eqnarray}\label{fv}
\left\{
  \begin{array}{lll}
    \hbox{Find}\ \psi_\eps\in\mathcal{X}_\Gamma \ \hbox{such that},\\
    \mathcal{A}_\eps(\psi_\eps,v)=l_\eps(v),\ \ \forall v\in\mathcal{X}_\Gamma,
  \end{array}
\right.
\end{eqnarray}
where the functional space $\mathcal{X}_\Gamma,$ the bilinear form
$\mathcal{A}_\eps,$ and the linear form $l_\eps$ are defined by
\begin{align}
\mathcal{X}_\Gamma&=\left\{v\in H^1(\Omega)\ \hbox{such that}\
\hbox{div}\ v=0 \ \hbox{and}\ v=0\ \hbox{on}\ \Gamma \right\},\\
\label{biliniearform}\mathcal{A}_\eps(\psi_\eps,v)&=\int_{\Omega}\nu\nabla\psi_\eps:\nabla v\ \dx+\int_{\Omega}(\alpha+\delta c_\eps)\psi_\eps. v\ \dx,\\
l_\eps(v)&=\int_\Sigma g.  v\ \ds.\label{linearform}
\end{align}
With above statements, we deduce that the cost functional $\K$ is
then defined in the perturbed domain as
\begin{align}\label{Jnew}
\K({\O_{z,\eps}},\psi_\eps)=\mathcal{J}(\eps)=\int_\Omega\Big|\psi_\eps-\psi^d\Big|^2\dx+\rho\mathcal{P}(\O_{z,\varepsilon},\Omega).
\end{align}
In the particular case $\omega_{z,\eps}=\emptyset$ (i.e, $\eps=0$),
the cost functional $\K$ is defined by $L^2$-norm without the
regularization term:
\begin{align}\label{Jnew0}
\K(\emptyset,\psi_0)=\mathcal{J}(0)=\int_\Omega\Big|\psi_0-\psi^d\Big|^2\dx,
\end{align}
where $\psi_0$ is the solution to
\begin{equation}\label{00}
\left\{
\begin{array}{c}
\begin{array}{r l l l}
-\nu\Delta \psi_0+\alpha \psi_0+\nabla p_0 & = 0& \mbox{in } & { \Omega,} \\%[0.12cm]
 \mbox{div } \psi_0 & =  0 & \mbox{in } &{ \Omega,}\\%[0.12cm]
  \psi_0& =0 &  \mbox {on }&{  \Gamma,}\\%[0.12cm]
  \sigma(\psi_0,p_0)\textbf{n}& =g &  \mbox {on }&{  \Sigma.}
\end{array}
\end{array}
 \right.
\end{equation}

The main objective of the following consists in establishing an
asymptotic expansion for $\mathcal{J}$ in order to determine the
location and shape of $\O.$ Before that, we need the following
preliminary lemmas.

%In order to prove this theorem, we will first establish the
%following lemma.
\begin{lemma}\label{lemma3}
Let $\psi_\eps$ and $\psi_0$ be the solutions to the problems
(\ref{22}) and (\ref{00}), respectively. Then, there exists a
positive constant $c$ independent of $\varepsilon$ such that
$$
\Big\|\psi_\eps-\psi_0\Big\|_{H^1(\Omega)}\leq c\ \eps^{1+\tau},
$$
for any $0<\tau<1.$
\end{lemma}

\begin{proof}
From (\ref{22}) and (\ref{00}) and using Green's formula, we obtain
\begin{align}
\nonumber\int_{\Omega}\nu\nabla\Big(\psi_\eps-\psi_0\Big):\nabla v\
\dx&+\int_{\Omega}(\alpha+\delta c_\eps)\Big(\psi_\eps-\psi_0\Big).
v\
\dx\\
&\ \ \ +\int_{\Omega}\delta c_\eps\psi_0. v\ \dx=0\ \forall
v\in\mathcal{X}_\Gamma.\label{al1}
\end{align}
By taking $v=\psi_\eps-\psi_0$ in (\ref{al1}) as a test function, we
get
\begin{align*}
\int_{\Omega}\nu\Big|\nabla\Big(\psi_\eps-\psi_0\Big)\Big|^2
\dx+\int_{\Omega}(\alpha+\delta c_\eps)\Big|\psi_\eps-\psi_0\Big|^2
\dx=-\int_{\omega_{z,\eps}}\psi_0. \Big(\psi_\eps-\psi_0\Big) \dx.
\end{align*}
From the Cauchy-Schwarz inequality and the smoothness of $\psi_0$ in
$\omega_{z,\eps},$ there exists a positive constant $c_1$
independent of $\eps$ such that
\begin{align*}
\int_{\Omega}\nu\Big|\nabla\Big(\psi_\eps-\psi_0\Big)\Big|^2
\dx+\int_{\Omega}(\alpha+\delta c_\eps)\Big|\psi_\eps-\psi_0\Big|^2
\dx&\leq\Big\|\psi_0\Big\|_{L^2(\omega_{z,\eps})}
\Big\|\psi_\eps-\psi_0\Big\|_{L^2(\omega_{z,\eps})}\\
&\leq c_1\eps \Big\|\psi_\eps-\psi_0\Big\|_{L^2(\omega_{z,\eps})}.
\end{align*}
Notice that, H\"{o}lder inequality and the Sobolev embedding theorem
can be used to derive
\begin{align*}
\Big\|\psi_\eps-\psi_0\Big\|_{L^2(\omega_{z,\eps})}\leq
c_2\eps^{1/q}\Big\|\psi_\eps-\psi_0\Big\|_{L^{2p}(\omega_{z,\eps})}\leq
c_3\eps^\tau\Big\|\psi_\eps-\psi_0\Big\|_{H^1(\Omega)},
\end{align*}
for any $1<q<\infty$ with $1/p+1/q=1.$ Let us denote $\tau=1/q$
which implies $0<\tau<1.$ Therefore,
\begin{align*}
\int_{\Omega}\nu\Big|\nabla\Big(\psi_\eps-\psi_0\Big)\Big|^2
\dx+\int_{\Omega}(\alpha+\delta c_\eps)\Big|\psi_\eps-\psi_0\Big|^2
\dx\leq c_4\eps^{\tau+1} \Big\|\psi_\eps-\psi_0\Big\|_{H^1(\Omega)}.
\end{align*}
On the other hand, we have
\begin{align*}
\min\{\nu,\alpha\}\Big\|\psi_\eps-\psi_0\Big\|^2_{H^1(\Omega)}\leq\int_{\Omega}\nu\Big|\nabla\Big(\psi_\eps-\psi_0\Big)\Big|^2
\dx+\int_{\Omega}(\alpha+\delta c_\eps)\Big|\psi_\eps-\psi_0\Big|^2
\dx.
\end{align*}
Therefore,
\begin{align*}
\Big\|\psi_\eps-\psi_0\Big\|_{H^1(\Omega)}\leq\frac{c_4
\varepsilon^{\tau+1}}{\min\{\nu,\alpha\}}=c\varepsilon^{\tau+1}\ \
\hbox{with}\ \ c=\frac{c_4}{\min\{\nu,\alpha\}}.
\end{align*}
\end{proof}

\begin{lemma}\label{proposition}The cost functional $\K$ is differential with respect to
$\psi_0,$ such that
\begin{align}\label{A1}
D\K(\emptyset,\psi_0)w=2\int_\Omega\Big(\psi_0-\psi^d\Big).w\ \dx\ \
\forall w\in H^1(\Omega)
\end{align}
and we have \begin{align}\label{A2}
\K(\O_{z,\eps},\psi_\eps)-\K(\emptyset,\psi_0)=D\K(\emptyset,\psi_0)(\psi_\eps-\psi_0)+o(\eps^2).
\end{align}
\end{lemma}
\begin{proof} The verification of the differentiability of $\K$ with respect to $\psi_0$ such that
\begin{align*}
D\K(\emptyset,\psi_0)w=2\int_\Omega\Big(\psi_0-\psi^d\Big).w\ \dx\ \
\forall w\in H^1(\Omega)
\end{align*}
is trivial.

By subtracting (\ref{Jnew0}) from (\ref{Jnew}), we have
\begin{align}\label{ASS}
\nonumber&\K({\O_{z,\eps}},\psi_\eps)-\K(\emptyset,\psi_0)=\int_\Omega\Big|\psi_\eps-\psi^d\Big|^2\dx-\int_\Omega\Big|\psi_0-\psi^d\Big|^2\dx+\rho\mathcal{P}(\O_{z,\eps},\Omega)\\
\nonumber&=\int_\Omega\Big|\Big(\psi_\eps-\psi_0\Big)+\Big(\psi_0-\psi^d\Big)\Big|^2\dx-\int_\Omega\Big|\psi_0-\psi^d\Big|^2\dx+\rho\mathcal{P}(\O_{z,\eps},\Omega)\\
\nonumber&=2\int_\Omega\Big(\psi_\eps-\psi_0\Big).\Big(\psi_0-\psi^d\Big)\dx+\int_\Omega\Big|\psi_\eps-\psi_0\Big|^2\dx+\rho\mathcal{P}(\O_{z,\eps},\Omega)\\
&=D\K(\emptyset,\psi_0)(\psi_\eps-\psi_0)+\int_\Omega\Big|\psi_\eps-\psi_0\Big|^2\dx+\rho\mathcal{P}(\O_{z,\eps},\Omega),
\end{align}
where
$$D\K(\emptyset,\psi_0)(\psi_\eps-\psi_0)=2\int_\Omega\Big(\psi_\eps-\psi_0\Big).\Big(\psi_0-\psi^d\Big)\dx.$$
Using Lemma \ref{lemma3}, the second term on the right-hand-side of
the equality (\ref{ASS}) admits the following estimate
$$
\int_\Omega\Big|\psi_\eps-\psi_0\Big|^2\dx=o(\varepsilon^2).
$$
To estimate the last term in the right-hand-side of (\ref{ASS}), we
need to take $\rho=\varepsilon^3$ then,
\begin{align*}
\rho\mathcal{P}(\O_{z,\varepsilon},\Omega)=o(\varepsilon^2),
\end{align*}
since $\mathcal{P}(\O_{z,\varepsilon},\Omega)<+\infty.$ Therefore;
\begin{align*}
\K({\O_{z,\eps}},\psi_\eps)-\K(\emptyset,\psi_0)=D\K(\emptyset,\psi_0)(\psi_\eps-\psi_0)+o(\eps^2).
\end{align*}
\end{proof}

Now, we are ready to state our main result of this section.

\begin{theorem}\label{theorem2.4} Let
$\O_{z,\varepsilon}=z+\varepsilon\O$ be a small obstacle in the
fluid flow domain $\Omega$ and let $\mathcal{J}$ be a cost function
of the form
\begin{align*}
\mathcal{J}(\eps)=\int_\Omega\Big|\psi_\eps-\psi^d\Big|^2\dx+\rho\mathcal{P}(\O_{z,\varepsilon},\Omega).
\end{align*}
Then the cost function $\mathcal{J}$ has the following asymptotic
expansion:
$$
\mathcal{J}(\eps)-\mathcal{J}(0)=k|\O|\eps^2\mathcal{G}(z)+o(\eps^2),
$$
where $|\O|$ is the Lebesgue measure (volume) of $\O$ and
$\mathcal{G}$ is the topological gradient defined in $\Omega$ by
$$
\mathcal{G}(z)=\psi_0(z).\vartheta_0(z),
$$
with $\vartheta_0$ is the solution to the adjoint problem: find
$\vartheta_0\in\mathcal{X}_\Gamma$ such that
\begin{align}\label{adjoint}
 \mathcal{A}_0(w,\vartheta_0)=-D\K(\emptyset,\psi_0)w\ \ \forall
w\in\mathcal{X}_\Gamma.
\end{align}
\end{theorem}

\begin{proof}Let us consider the Lagrangian $\mathcal{L}_\varepsilon$ defined by
$$
\mathcal{L}_\eps(u,v)=\K(\O_{z,\varepsilon},u)+\mathcal{A}_\eps(u,v)-l_\eps(v)\
\ \forall u,\ v\in\mathcal{X}_\Gamma.
$$
By setting $u=\psi_\varepsilon$ in the above equality and using that
$\psi_\varepsilon$ is the weak solution to (\ref{fv}), we obtain
$$
\mathcal{L}_\eps(\psi_\eps,v)=\K(\O_{z,\eps},\psi_\eps)\ \ \forall
v\in\mathcal{X}_\Gamma.
$$
Hence,
\begin{align}\label{VL}
\nonumber&\mathcal{J}(\eps)-\mathcal{J}(0)=\mathcal{L}_\eps(\psi_\eps,v)-\mathcal{L}_0(\psi_0,v)\\
&=\K(\O_{z,\eps},\psi_\eps)-\K(\emptyset,\psi_0)+\mathcal{A}_\eps(\psi_\eps,v)-\mathcal{A}_0(\psi_0,v)+l_0(v)-l_\eps(v).
\end{align}
The linear form $l$ is independent of $\eps,$ then
\begin{align}\label{variationl}
l_0(v)-l_\eps(v)=0,\ \ \forall v\in\mathcal{X}_\Gamma.
\end{align}
For all $v\in\mathcal{X}_\Gamma,$ the variation of the bilinear form
is given by
\begin{align*}
&\mathcal{A}_\eps(\psi_\eps,v)-\mathcal{A}_0(\psi_0,v)\\
&=\int_\Omega\nu\nabla\Big(\psi_\eps-\psi_0\Big):\nabla v
\dx+\int_\Omega\alpha\Big(\psi_\eps-\psi_0\Big).v
\dx+\int_{\O_{z,\eps}}k\psi_\eps.v \dx\\
&=\mathcal{A}_0(\psi_\eps-\psi_0,v)+\int_{\O_{z,\eps}}k\psi_\eps.v
\dx.
\end{align*}
Choosing $v=\vartheta_0$ in the above equality, where $\vartheta_0$
is solution to (\ref{adjoint}), we obtain
\begin{align*}
\mathcal{A}_\eps(\psi_\eps,\vartheta_0)-\mathcal{A}_0(\psi_0,\vartheta_0)=\mathcal{A}_0(\psi_\eps-\psi_0,\vartheta_0)+\int_{\O_{z,\eps}}k\psi_\eps.\vartheta_0\
\dx.
\end{align*}
By taking $w=\psi_\eps-\psi_0$ as a test function in
(\ref{adjoint}), we deduce that
\begin{align}\label{variationA}
\mathcal{A}_\eps(\psi_\eps,\vartheta_0)-\mathcal{A}_0(\psi_0,\vartheta_0)=-D\K(\emptyset,\psi_0)(\psi_\eps-\psi_0)+\int_{\O_{z,\eps}}k\psi_\eps.\vartheta_0\
\dx.
\end{align}
Then, it follows from (\ref{VL}), (\ref{variationl}),
(\ref{variationA}) and (\ref{A2}) that
\begin{align}
\mathcal{J}(\eps)-\mathcal{J}(0)=\int_{\O_{z,\eps}}k\psi_\eps.\vartheta_0\
\dx+o(\eps^2).
\end{align}
Now we prove that
$$
\int_{\O_{z,\eps}}k\psi_\eps.\vartheta_0\
\dx=k|\O|\eps^2\psi_0(z).\vartheta_0(z)+o(\eps^2).
$$
We have
\begin{align}\label{ww}
\int_{\O_{z,\eps}}k\psi_\eps.\vartheta_0\
\dx=k\int_{\O_{z,\eps}}\psi_0.\vartheta_0\
\dx+k\int_{\O_{z,\eps}}\Big(\psi_\eps-\psi_0\Big).\vartheta_0\ \dx.
\end{align}
Let us first focus on the first term in the right-hand side of
(\ref{ww}). Using the change of variable $x=z+\eps y$, we derive
\begin{align*}
k\int_{\O_{z,\eps}}\psi_0.\vartheta_0
\dx=k\eps^2&\int_{\O}\psi_0(z).\vartheta_0(z)
\dy\\
&+k\varepsilon^2\int_{\O}\Big[\psi_0(z+\varepsilon
y).\vartheta_0(z+\varepsilon y)-\psi_0(z).\vartheta_0(z)\Big]\dy.
\end{align*}
By the Taylor expansion and using the fact that $\psi_0$ and
$\vartheta_0$ are regular near $z,$ we deduce that
\begin{align*}
k\eps^2\int_{\O}\Big[\psi_0(z+\varepsilon
y).\vartheta_0(z+\varepsilon y)-\psi_0(z).\vartheta_0(z)\Big]\
\dy=O(\varepsilon^3).
\end{align*}
Hence,
\begin{align*}
k\int_{\O_{z,\eps}}\psi_0.\vartheta_0\
\dx=k|\O|\eps^2\psi_0(z).\vartheta_0(z)+o(\eps^2).
\end{align*}

For the other term in the right-hand side of (\ref{ww}) using
H\"{o}lder inequality, we derive
\begin{align*}
\Big|\int_{\O_{z,\eps}}\Big(\psi_\eps-\psi_0\Big).\vartheta_0\
\dx\Big|&\leq
\Big\|\vartheta_0\Big\|_{\mathrm{L}^2(\O_{z,\eps})}\Big\|\psi_\eps-\psi_0\Big\|_{\mathrm{L}^2(\O_{z,\eps})}\\
&\leq
\Big\|\vartheta_0\Big\|_{\mathrm{L}^2(\O_{z,\eps})}\Big\|\psi_\eps-\psi_0\Big\|_{\mathrm{H}^1(\Omega)}.
\end{align*}
We know using elliptic regularity that $\vartheta_0$ is uniformly
bounded in $\O_{z,\eps}$. Thus
$$
\Big\|\vartheta_0\Big\|_{\mathrm{L}^2(\O_{z,\eps})}^2=\int_{\O_{z,\eps}}\Big|\vartheta_0\Big|^2\
\dx\leq c\int_{\O}\eps^2=O(\eps^2).
$$
Consequently,
\begin{align*}
\Big|\int_{\O_{z,\eps}}\Big(\psi_\eps-\psi_0\Big).\vartheta_0\
\dx\Big|\leq c\
\eps\Big\|\psi_\eps-\psi_0\Big\|_{\mathrm{H}^1(\Omega)}.
\end{align*}
From Lemma \ref{lemma3}, we deduce that
\begin{align*}
\Big|\int_{\O_{z,\eps}}\Big(\psi_\eps-\psi_0\Big).\vartheta_0\
\dx\Big|\leq c\ \eps^{2+\tau}=o(\eps^2)
\end{align*}
and the proof is completed.
\end{proof}

%%%%%%%%%%%%%%%%%%%%%%%%%%%%%%%%%%%%%%%%%%%%%%%%%%%%%%%%%%%%%%%%
\section{Numerical results}\label{numerique}
%%%%%%%%%%%%%%%%%%%%%%%%%%%%%%%%%%%%%%%%%%%%%%%%%%%%%%%%%%%%%%%%
In this section, we present some numerical tests showing the
efficiency of the proposed method. The aim is to reconstruct the
location and shape of an unknown obstacle $\O$ inserted inside the
fluid flow domain $\Omega$ from internal data by using a level-set
curve of the topological gradient.

From Theorem \ref{theorem2.4}, the functional $\K$ has the following
topological asymptotic expansion:
\begin{align}\label{asyyy}
\K({\O_{z,\eps}},\psi_\eps)=\K(\emptyset,\psi_0)+\varepsilon^2k\mathcal{G}(z)+o(\eps^2),\end{align}
with the function $z\mapsto \mathcal{G}(z)$ is the topological
gradient defined by \begin{align}\label{gradient}
\mathcal{G}(z)=\psi_0(z).\vartheta_0(z),
\end{align}
where $\psi_0$ is the solution to (\ref{00}) and $\vartheta_0$ is
the solution of the adjoint problem
\begin{equation}\label{vv00}
\left\{
\begin{array}{c}
\begin{array}{r l l l}
-\nu\Delta \vartheta_0+\alpha \vartheta_0+\nabla p_0 & = -2(\psi_0-\psi^d)& \mbox{in } & { \Omega,} \\%[0.12cm]
 \mbox{div } \vartheta_0 & =  0 & \mbox{in } &{ \Omega,}\\%[0.12cm]
  \vartheta_0& =0 &  \mbox {on }&{  \Gamma,}\\%[0.12cm]
  \sigma(\vartheta_0,p_0)\textbf{n}& =0 &  \mbox {on }&{  \Sigma.}
\end{array}
\end{array}
 \right.
\end{equation}
The asymptotic expansion (\ref{asyyy}) motivates the reconstruction
technique: if we place obstacles in the zone where the topological
gradient $\mathcal{G}$ takes pronounced negative values, the error
function is expected to decrease, yielding a prediction of the
location, shape, size, and number of the obstacles. We identify then
a guess for $\Omega$ by considering the set
$$
\O_\gamma=\Big\{x\in\Omega;\ \
\mathcal{G}(x)\leq(1-\gamma)\min_{x\in\Omega}\mathcal{G}(x) \Big\},
$$
where $\mathcal{G}$ is defined by (\ref{gradient}) and
$0\leq\gamma\leq1$ is a constant that can be tuned. Therefore, the
unknown obstacle $\O$ is likely to be located at the zone where the
topological gradient $\mathcal{G}$ is the most negative. To make the
numerical simulations presented here, we use a $\mathbb{P}_2$ finite
elements discretization to solve the direct problem (\ref{00}) and
the adjoint problem (\ref{vv00}). The proposed numerical algorithm
is
based on the following main steps.\\
\textbf{One-shot algorithm.} \\
-Solve the direct problem (\ref{00}) and the adjoint problem
(\ref{vv00}).\\
-Compute the topological gradient $\mathcal{G}=\psi_0.\vartheta_0$ in $\Omega.$\\
-Reconstruct the unknown obstacle $\O.$

The location of $\O$ is given by the point $z\in\Omega$ where the
topological gradient $\mathcal{G}$ is most negative (i.e,
$z=\displaystyle \operatorname*{arg\,
min}_{x\in\Omega}\mathcal{G}(x)$). The size of $\O$ is approximated
as follows
$$
\O=\Big\{x\in\Omega;\ \
\mathcal{G}(x)\leq(1-\gamma^*)\min_{x\in\Omega}\mathcal{G}(x)
\Big\},
$$
where $\gamma^*\in(0,1)$ such that
$$\mathrm{J}(\O_{\gamma^*})\leq\mathrm{J}(\O_{\gamma})\
\forall \gamma\in(0,1),$$ with $ \O_{\gamma^*}=\Big\{x\in\Omega;\ \
\mathcal{G}(x)\leq(1-\gamma^*)\displaystyle\min_{x\in\Omega}\mathcal{G}(x)
\Big\}.$\\
The above topological gradient algorithm is classical and has been
used to solve various problems
\cite{carpio2019topological,le2019detection,hrizi2018one,amstutz2005crack,abda2009topological,ferchichi2013detection,hrizi2018new,burger2004incorporating,fulmanski2008level,jleli2015topological}
and so on.

In this paper, we extend this approach for the reconstruction of
obstacle.

\begin{remark}
In the particular case when the  exact obstacle $\O$ is known, the
best value $\gamma^*$ of the parameter $\gamma$ can be determined as
the minimum of the following error functional,
\begin{align}\label{err}
E(\gamma)=\Big[meas(\O\cup \O_{\gamma})-meas(\O\cap
\O_{\gamma})\Big]/meas(\O),\ \ \forall \gamma\in(0,1),
\end{align}
where $meas(H)$ is the Lebesgue measure of the set
$H\subset\mathbb{R}^2.$
\end{remark}

In all numerical tests, we use synthetic data, i.e., the measurement
$\psi^d$ is generated by numerically solving the problem (\ref{11}).
The square domain $\Omega=(0,1)\times(0,1)$ is used as a mould
filled with a viscous and incompressible fluid. For the boundary
$\Sigma$ and $\Gamma,$ see the sketch in Figure \ref{Figure1}.

The approximated solutions are computed using a uniform mesh with
$h=1/100.$ The numerical procedure is implemented using the free
software $FreeFem++.$

Next, we present some reconstruction results showing the efficiency
of the proposed one-shot algorithm.

%%%%%%%%%%%%%%%%%%%%%%%%%%%%%%%%%%%%%%%%%%%%%%%%%%%%%%%%%%%%%%%%%%%%%%%%%%%%%%%%%%%%%%%%%%%%%%%%%%%%%%
\subsection{Reconstruction of some obstacles}
%%%%%%%%%%%%%%%%%%%%%%%%%%%%%%%%%%%%%%%%%%%%%%%%%%%%%%%%%%%%%%%%%%%%%%%%%%%%%%%%%%%%%%%%%%%%%%%%%%%%%%
In this section, we study the reconstruction of obstacle having
circular or elliptical shapes with no noise added to the simulated
data.

\textbf{Example 1:} {\it Reconstruction of a circular-shaped
obstacle.} In this example, we test our procedure to detect an
obstacle having circular-shaped. More precisely, we want to
reconstruct an obstacle $\omega$ described by a disc centered at
$z=(0.5,0.5)$ with radius $r=0.05.$ The obtained reconstruction
results are illustrated in Figure \ref{Figure2}.

As one can observe in Figure \ref{Figure2}, the unknown obstacle
(see Figure \ref{Figure2}(b) circle black line ) is located in the
region where the topological gradient $\mathcal{G}$ is the most
negative (see Figure \ref{Figure2}(a) red zone) and it is
approximated by a level-set curve of the topological gradient (see
Figure \ref{Figure2}(b) red lines). The result is efficient and the
reconstruction of circular shape is very close to the actual
obstacle. We also observe the most negative values of the
topological gradient which are located near the actual boundary
$\partial\omega.$

To reconstruct the exact shape of the actual obstacle in Figure
\ref{Figure2}(b) (circle centered at $(0.5,0.5)$ with radius
$r=0.05$ (see Figure \ref{Figure2}(b) black line)), we minimize the
function $E$ and we take $\gamma^*$ the minimum of $E.$ To compute
numerically an approximation of the minimum of the function $E,$ we
divide the interval $(0, 1)$ into $\ell$ equal subintervals (i.e.,
of size $1/\ell$). We denote by $\gamma_i=i/\ell,\ 1\leq i\leq\ell$
the ($\ell+1$) endpoints of these intervals and we take
$\gamma^*=\operatorname*{arg\,
min}_{\gamma\in\{\gamma_1,...,\gamma_\ell\}}E(\gamma)$. We represent
the results in Figure \ref{Figure3} where the exact boundary
$\partial\omega$ represented in black and the obtained shape in red
(see Figure \ref{Figure3}(b)).

\textbf{Example 2:}  {\it Reconstruction of ellipse-shaped
obstacle.} In this example, we reconstruct an obstacle described by
an ellipse centered at $(0.5,0.5).$ We present the reconstruction
results in Figure \ref{Figure4} where we see that the boundary of
obstacle (see Figure \ref{Figure4}(a)) is again detected and located
in the zone where the topological gradient is negative (see Figure
\ref{Figure4}(a) red lines) and it is approximated by the set
$\omega_{\gamma^*}$ where $\gamma^*=0.15.$ Through this test, we
show that the proposed method is able to reconstruct approximately
the shape and location of obstacle described by an ellipse.

In conclusion of these simulations, this approach permits to give us
acceptable knowledge of the location and shape of obstacle having
circular or elliptical-shaped. The computation of the topological
gradient depends on the size of the obstacle. This remark is
illustrated by the following experiments.

%%%%%%%%%%%%%%%%%%%%%%%%%%%%%%%%%%%%%%%%%%%%%%%%%%%%%%%%%%%%%%%%%%%%%%%%%%%%%%%%%%%%%%%%%%%%%%
\subsection{Influence of the size of the obstacle}
%%%%%%%%%%%%%%%%%%%%%%%%%%%%%%%%%%%%%%%%%%%%%%%%%%%%%%%%%%%%%%%%%%%%%%%%%%%%%%%%%%%%%%%%%%%%%%
We now want to study how the size of an obstacle modifies the
quality of the detection given by our algorithm. In order to do
that, we test how is the detection of a single circle while we
increase the radius. For this test, we consider the circle centered
at $(0.5\,0.5)$ with radius $r\in\{0.03,0.06,0.12,0.18\}.$ The
obtained detection results are shown in Figure \ref{Figure5}.

From these results, we can notice that, when the obstacle is
relatively small, the reconstruction is quite efficient (see Figure
\ref{Figure5}(a)-(b)), but the quality is decreasing when the
obstacle becomes ``too big" (see Figure \ref{Figure5}(b)-(d)).

Next, we investigate the robustness of the method with respect to
noisy measurement.
%%%%%%%%%%%%%%%%%%%%%%%%%%%%%%%%%%%%%%%%%%%%%%%%%%%%%%%%%%%%%%%%%%%%%%%%%%%%%%%%%%
\subsection{Reconstruction results with noisy data}
%%%%%%%%%%%%%%%%%%%%%%%%%%%%%%%%%%%%%%%%%%%%%%%%%%%%%%%%%%%%%%%%%%%%%%%%%%%%%%%%%%
Now we are interested in investigating the robustness of the
reconstruction method when the measurement $\psi^d$ is corrupted
with Gaussian random noise. More precisely, the measurement $\psi^d$
is replaced by
$$
\psi^d_\rho=\psi^d+\delta\psi^d,
$$
where $\delta\psi^d$ is a Gaussian random noise with mean zero and
standard derivation $\delta\|\psi^d\|_{\infty},$ where $\delta$ is a
parameter.

For this test, we reconstruct an ellipse centered at $(0.5,0.5).$
The obtained results are illustrated in Figure \ref{Figure6}. From
the reconstruction results in Figure \ref{Figure6}, we can notice
that if the noise level no more than $20\%,$ that our algorithm is
able to detect the location and the shape of obstacle, whereas for a
noise level larger than $30\%$ the reconstruction becomes wrong.

%%%%%%%%%%%%%%%%%%%%%%%%%%%%%%%%%%%%%%%%%%%%%%%%%%%%%%%%%%%%%%%%%%%%%%%%
\section{Concluding remarks}\label{conclusion}
%%%%%%%%%%%%%%%%%%%%%%%%%%%%%%%%%%%%%%%%%%%%%%%%%%%%%%%%%%%%%%%%%%%%%%%%

The presented paper concerns the reconstruction of obstacle immersed
in a fluid governed by the Stokes-Brinkmann system in a
two-dimensional bounded domain $\Omega$ from internal data. In
particular, a non-iterative reconstruction method for solving the
above inverse problem has been proposed. The general idea consists
of rewriting the inverse problem as a topology optimization problem,
where a least square functional measuring the misfit between the
internal data measurements and the solution obtained from the model
is expanded (Stokes-Brinkmann system). The existence and the
stability of the optimization problem are proved. We have computed
the asymptotic expansion of the cost function using the penalization
technique without using the truncation method. The efficiency and
accuracy of the reconstruction algorithm are illustrated by some
numerical results. The presented method is general and can be
adapted for various inverse problems.

In this paper, focused on the topological sensitivity analysis and a
non-iterative reconstruction method, several mathematical issues of
high interest could not be discussed. The identifiability problem
for Stokes-Brinkmann problem is still an open one, will be the
subject of a forthcoming work.

Reconstructing obstacle from partial interior observation of the
velocity is also an interesting problem to tackle, for several
causes may lead to such a situation, especially when zones of the
fluid flow domain are not accessible to measurements.

\section*{Acknowledgements}
The authors would like to thank Professors Maria-Luisa Rap\'{u}n for
many helpful suggestions they made and for the careful reading of
the manuscript.

%%%%%%%%%%%%%%%%%%%%%%%%%%%%%%%%%%%%%%%%%%%%%%%%%%%%%%%%%%%%%%%%%%%%%%%%%%%%%%%%%%%%%%%%%%%%%%%%%%%%%%%%%%%%%%%

\bibliographystyle{abbrv} % Use the "unsrtnat" BibTeX style for formatting the Bibliography
\bibliography{contact}

\begin{thebibliography}{10}

\bibitem{abda2009topological}
A.~B. Abda, M.~Hassine, M.~Jaoua, and M.~Masmoudi.
\newblock Topological sensitivity analysis for the location of small cavities
  in stokes flow.
\newblock {\em SIAM Journal on Control and Optimization}, 48(5):2871--2900,
  2009.

\bibitem{alves2004determination}
C.~J. Alves and A.~L. Silvestre.
\newblock On the determination of point-forces on a stokes system.
\newblock {\em Mathematics and computers in Simulation}, 66(4-5):385--397,
  2004.

\bibitem{ambrosio2000functions}
L.~Ambrosio, N.~Fusco, and D.~Pallara.
\newblock {\em Functions of bounded variation and free discontinuity problems},
  volume 254.
\newblock Clarendon Press Oxford, 2000.

\bibitem{amstutz2005topological}
S.~Amstutz.
\newblock The topological asymptotic for the navier-stokes equations.
\newblock {\em ESAIM: Control, Optimisation and Calculus of Variations},
  11(3):401--425, 2005.

\bibitem{amstutz2005crack}
S.~Amstutz, I.~Horchani, and M.~Masmoudi.
\newblock Crack detection by the topological gradient method.
\newblock {\em Control and Cybernetics}, 34(1):81--101, 2005.

\bibitem{attouch2006variational}
H.~Attouch, G.~Buttazzo, and G.~Michaille.
\newblock Variational analysis in sobolev and bv spaces: applications to pdes
  and optimization.
\newblock 2006.

\bibitem{beretta2017size}
E.~Beretta, C.~Cavaterra, J.~H. Ortega, and S.~Zamorano.
\newblock Size estimates of an obstacle in a stationary stokes fluid.
\newblock {\em Inverse Problems}, 33(2):025008, 2017.

\bibitem{borrvall2003topology}
T.~Borrvall and J.~Petersson.
\newblock Topology optimization of fluids in stokes flow.
\newblock {\em International journal for numerical methods in fluids},
  41(1):77--107, 2003.

\bibitem{burger2004incorporating}
M.~Burger, B.~Hackl, and W.~Ring.
\newblock Incorporating topological derivatives into level set methods.
\newblock {\em Journal of Computational Physics}, 194(1):344--362, 2004.

\bibitem{carpio2019topological}
A.~Carpio, T.~G. Dimiduk, F.~Le~Lou{\"e}r, and M.~L. Rap{\'u}n.
\newblock When topological derivatives met regularized gauss-newton iterations
  in holographic 3d imaging.
\newblock {\em Journal of Computational Physics}, 2019.

\bibitem{caubet2015detection}
F.~Caubet, C.~Conca, and M.~Godoy.
\newblock On the detection of several obstacles in 2d stokes flow: topological
  sensitivity and combination with shape derivatives.
\newblock 2015.

\bibitem{caubet2012localization}
F.~Caubet and M.~Dambrine.
\newblock Localization of small obstacles in stokes flow.
\newblock {\em Inverse Problems}, 28(10):105007, 2012.

\bibitem{ferchichi2013detection}
J.~Ferchichi, M.~Hassine, and H.~Khenous.
\newblock Detection of point-forces location using topological algorithm in
  stokes flows.
\newblock {\em Applied Mathematics and Computation}, 219(12):7056--7074, 2013.

\bibitem{fulmanski2008level}
P.~Fulmanski, A.~Laurain, J.-F. Scheid, and J.~Soko{\l}owski.
\newblock Level set method with topological derivatives in shape optimization.
\newblock {\em International Journal of Computer Mathematics},
  85(10):1491--1514, 2008.

\bibitem{garreau2001topological}
S.~Garreau, P.~Guillaume, and M.~Masmoudi.
\newblock The topological asymptotic for pde systems: the elasticity case.
\newblock {\em SIAM journal on control and optimization}, 39(6):1756--1778,
  2001.

\bibitem{girault2012finite}
V.~Girault and P.-A. Raviart.
\newblock {\em Finite element methods for Navier-Stokes equations: theory and
  algorithms}, volume~5.
\newblock Springer Science \& Business Media, 2012.

\bibitem{guillaume2004topological}
P.~Guillaume and K.~S. Idris.
\newblock Topological sensitivity and shape optimization for the stokes
  equations.
\newblock {\em SIAM Journal on Control and Optimization}, 43(1):1--31, 2004.

\bibitem{hassine2004topological}
M.~Hassine and M.~Masmoudi.
\newblock The topological asymptotic expansion for the quasi-stokes problem.
\newblock {\em ESAIM: Control, Optimisation and Calculus of Variations},
  10(4):478--504, 2004.

\bibitem{heck2007reconstruction}
H.~Heck, G.~Uhlmann, and J.-N. Wang.
\newblock Reconstruction of obstacles immersed in an incompressible fluid.
\newblock {\em Inverse Problems and Imaging}, 1(1):63--76, 2007.

\bibitem{hrizi2018one}
M.~Hrizi and M.~Hassine.
\newblock One-iteration reconstruction algorithm for geometric inverse source
  problem.
\newblock {\em Journal of Elliptic and Parabolic Equations}, 4(1):177--205,
  2018.

\bibitem{hrizi2018new}
M.~Hrizi, M.~Hassine, and R.~Malek.
\newblock A new reconstruction method for a parabolic inverse source problem.
\newblock {\em Applicable Analysis}, pages 1--33, 2018.

\bibitem{jleli2015topological}
M.~Jleli, B.~Samet, and G.~Vial.
\newblock Topological sensitivity analysis for the modified helmholtz equation
  under an impedance condition on the boundary of a hole.
\newblock {\em Journal de Math{\'e}matiques Pures et Appliqu{\'e}es},
  103(2):557--574, 2015.

\bibitem{koster2007numerical}
D.~K{\"o}ster.
\newblock Numerical simulation of acoustic streaming on surface acoustic
  wave-driven biochips.
\newblock {\em SIAM Journal on Scientific Computing}, 29(6):2352--2380, 2007.

\bibitem{krotkiewski2011importance}
M.~Krotkiewski, I.~S. Ligaarden, K.-A. Lie, and D.~W. Schmid.
\newblock On the importance of the stokes-brinkman equations for computing
  effective permeability in karst reservoirs.
\newblock {\em Communications in Computational Physics}, 10(5):1315--1332,
  2011.

\bibitem{le2019detection}
F.~Le~Lou{\"e}r and M.-L. Rap{\'u}n.
\newblock Detection of multiple impedance obstacles by non-iterative
  topological gradient based methods.
\newblock {\em Journal of Computational Physics}, 2019.

\bibitem{lechleiter2013factorization}
A.~Lechleiter and T.~Rienm{\"u}ller.
\newblock Factorization method for the inverse stokes problem.
\newblock {\em Inverse Problems and Imaging}, 7:1271--1293, 2013.

\bibitem{masmoudi2002topological}
M.~Masmoudi.
\newblock The topological asymptotic, computational methods for control
  applications, ed. h. kawarada and j. p{\'e}riaux.
\newblock {\em International Series, Gakuto}, 2002.

\bibitem{masmoudi2005topological}
M.~Masmoudi, J.~Pommier, and B.~Samet.
\newblock The topological asymptotic expansion for the maxwell equations and
  some applications.
\newblock {\em Inverse Problems}, 21(2):547, 2005.

\bibitem{novotny2012topological}
A.~A. Novotny and J.~Soko{\l}owski.
\newblock {\em Topological derivatives in shape optimization}.
\newblock Springer Science \& Business Media, 2012.

\bibitem{pommier2004topological}
J.~Pommier and B.~Samet.
\newblock The topological asymptotic for the helmholtz equation with dirichlet
  condition on the boundary of an arbitrarily shaped hole.
\newblock {\em SIAM journal on control and optimization}, 43(3):899--921, 2004.

\bibitem{popov2009multiscale}
P.~Popov, Y.~Efendiev, and G.~Qin.
\newblock Multiscale modeling and simulations of flows in naturally fractured
  karst reservoirs.
\newblock {\em Communications in computational physics}, 6(1):162, 2009.

\bibitem{samet2003topological}
B.~Samet, S.~Amstutz, and M.~Masmoudi.
\newblock The topological asymptotic for the helmholtz equation.
\newblock {\em SIAM Journal on Control and Optimization}, 42(5):1523--1544,
  2003.

\bibitem{schumacher1996topologieoptimierung}
A.~Schumacher.
\newblock {\em Topologieoptimierung von Bauteilstrukturen unter Verwendung von
  Lochpositionierungskriterien}.
\newblock PhD thesis, Forschungszentrum f{\"u}r Multidisziplin{\"a}re Analysen
  und Angewandte Strukturoptimierung. Institut f{\"u}r Mechanik und
  Regelungstechnik, 1996.

\bibitem{sokolowski1999topological}
J.~Sokolowski and A.~Zochowski.
\newblock On the topological derivative in shape optimization.
\newblock {\em SIAM journal on control and optimization}, 37(4):1251--1272,
  1999.

\bibitem{yan2019shape}
W.~Yan, M.~Liu, and F.~Jing.
\newblock Shape inverse problem for stokes--brinkmann equations.
\newblock {\em Applied Mathematics Letters}, 88:222--229, 2019.

\end{thebibliography}
\newpage

\begin{figure}[!h]
\centering
    \includegraphics[width=90mm]{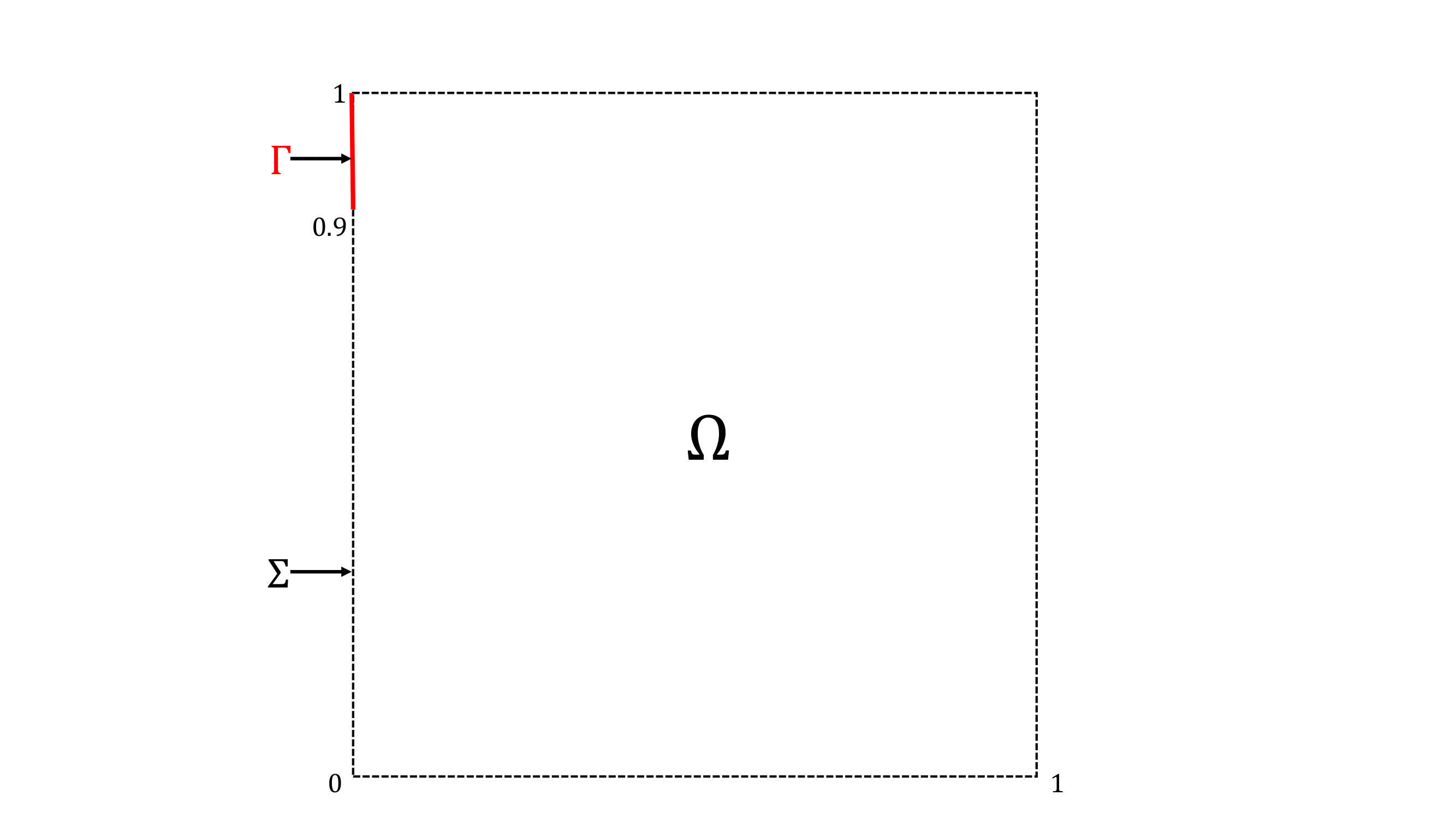}
    \caption{Domain $\Omega$ with boundary $\partial\Omega=\Sigma\cup\Gamma$}\label{Figure1}
    \end{figure}

\begin{figure} [!h]
     \centering
    \begin{tabular}{cc}
    \includegraphics[width=60mm]{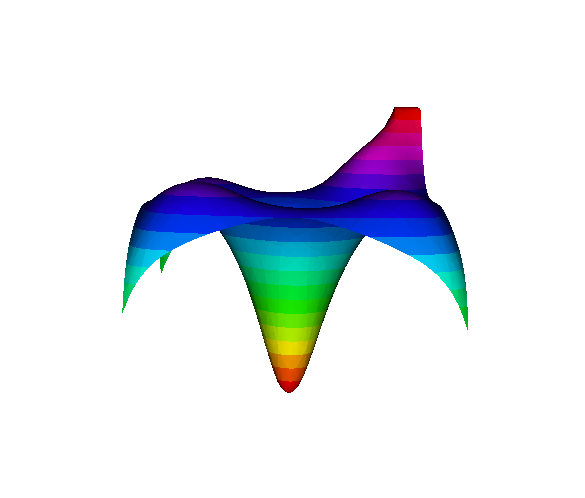}& \includegraphics[width=60mm]{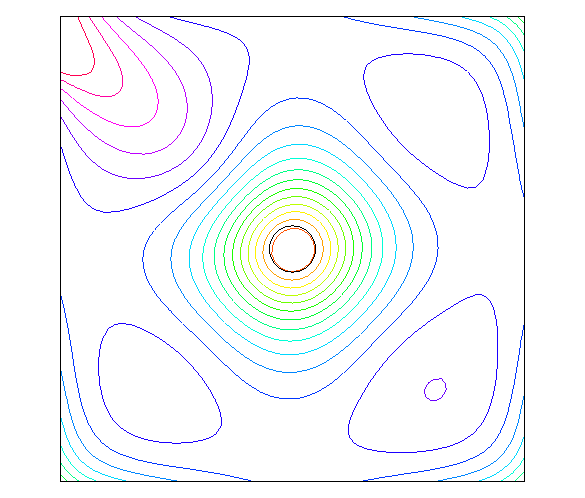}\\
(a) Negative zone (zed zone) of $\mathcal{G}$& (b) Iso-values of
$\mathcal{G}$
    \end{tabular}
    \caption{\label{Figure2} Topological gradient $\mathcal{G}$ in the presence of a circle shape}
    \end{figure}

\begin{figure} [!h]
     \centering
    \begin{tabular}{cc}
    \includegraphics[width=60mm]{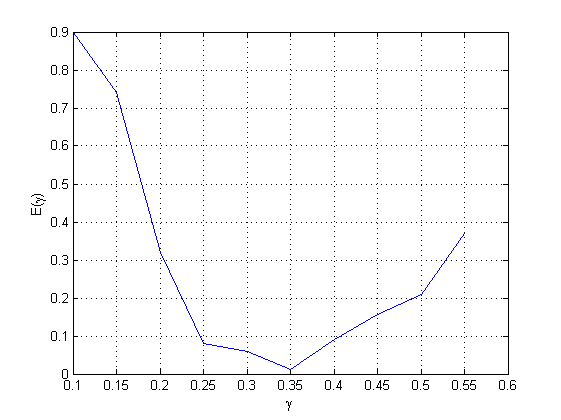}& \includegraphics[width=60mm]{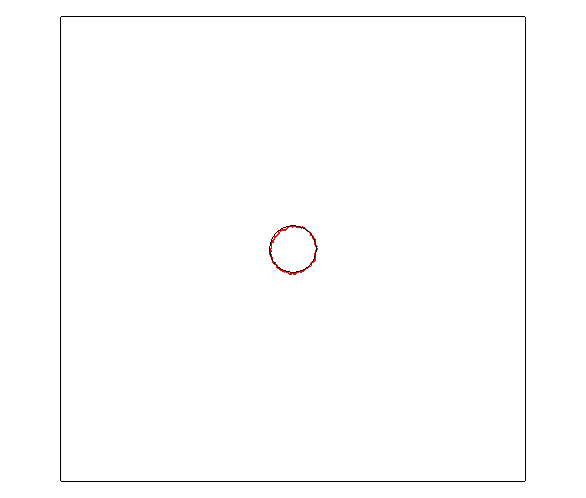}\\
(a)Variation of $E$ with respect to $\gamma$ & (b) Reconstruction
with $\gamma^*=0.35$
    \end{tabular}
    \caption{\label{Figure3} Reconstruction of obstacle having a circular-shaped}
    \end{figure}

\begin{figure} [!h]
     \centering
    \begin{tabular}{cc}
    \includegraphics[width=60mm]{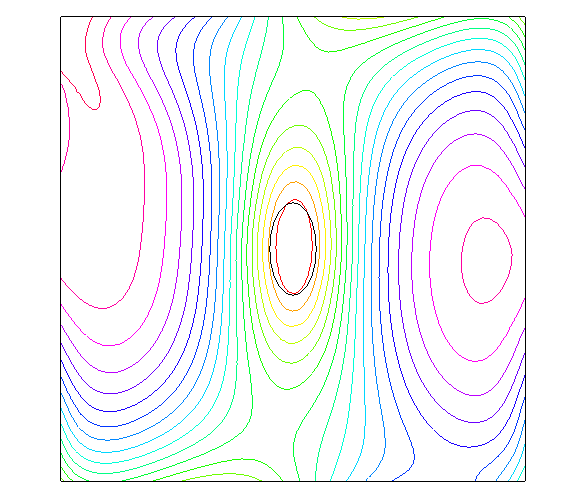}& \includegraphics[width=60mm]{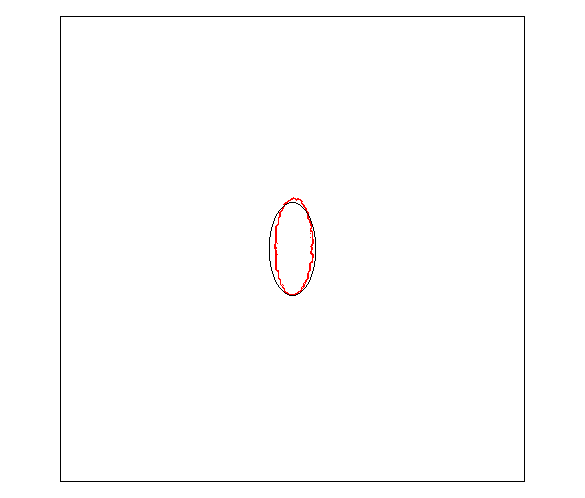}\\
(a)Iso-values of $\mathcal{G}$ & (b) Reconstruction with
$\gamma^*=0.62$
    \end{tabular}
    \caption{\label{Figure4} Reconstruction of obstacle having an ellipse-shaped}
    \end{figure}

\begin{figure}[!h]
     \centering
    \begin{tabular}{cc}
    \includegraphics[width=60mm]{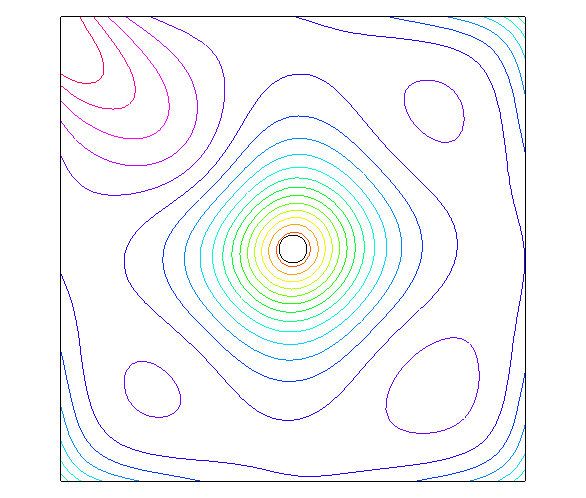}& \includegraphics[width=60mm]{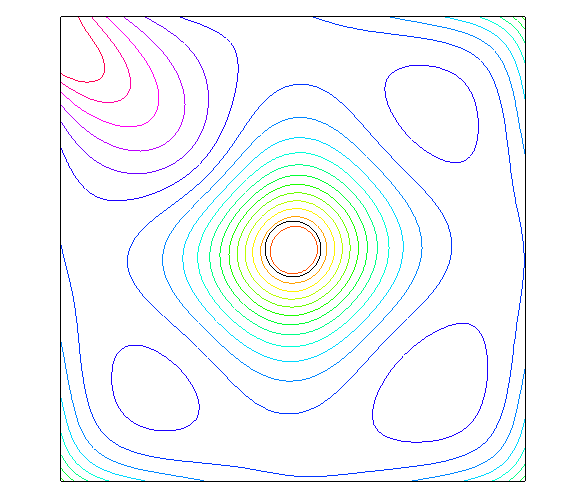}\\
    (a)Iso-values of $\mathcal{G}$ with $r=0.03$& (b)Iso-values of $\mathcal{G}$ with $r=0.06$\\
     \includegraphics[width=60mm]{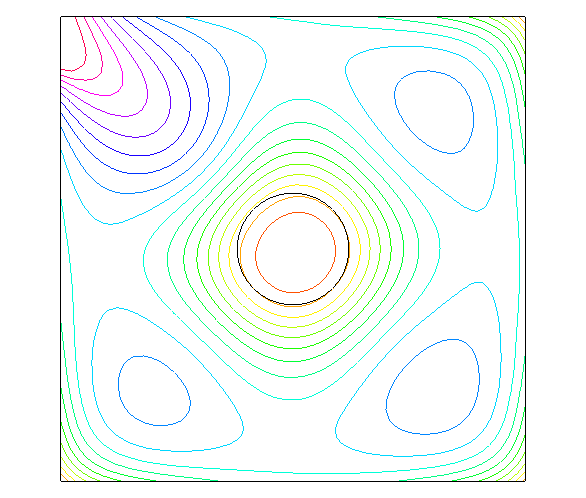}& \includegraphics[width=60mm]{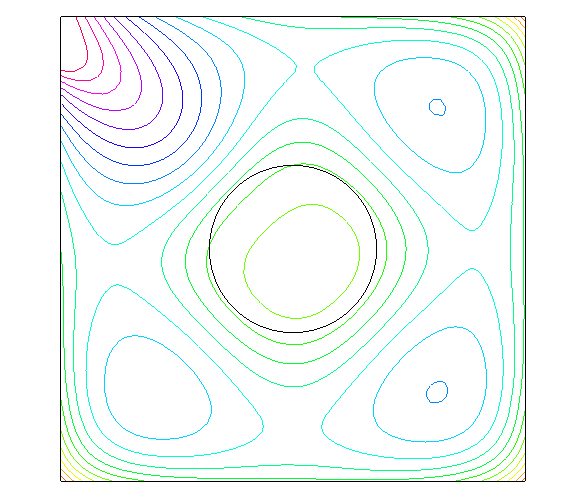}\\
    (c)Iso-values of $\mathcal{G}$ with $r=0.12$  & (d)Iso-values of $\mathcal{G}$ with $r=0.18$\\
    \end{tabular}
    \caption{\label{Figure5} Iso-values of the topological gradient when we increase the size of the object}
    \end{figure}

\begin{figure}[!h]
     \centering
    \begin{tabular}{cc}
    \includegraphics[width=60mm]{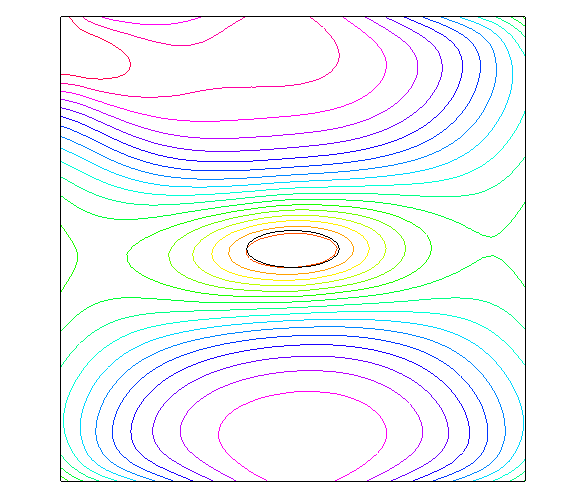}& \includegraphics[width=60mm]{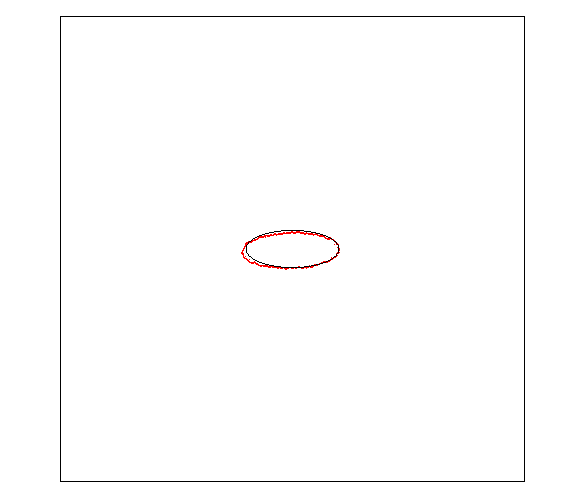}\\
    (a)Iso-values of $\mathcal{G}$ with $0\%$ noise & (b)Reconstruction with $\gamma^*=0.63$\\
     \includegraphics[width=60mm]{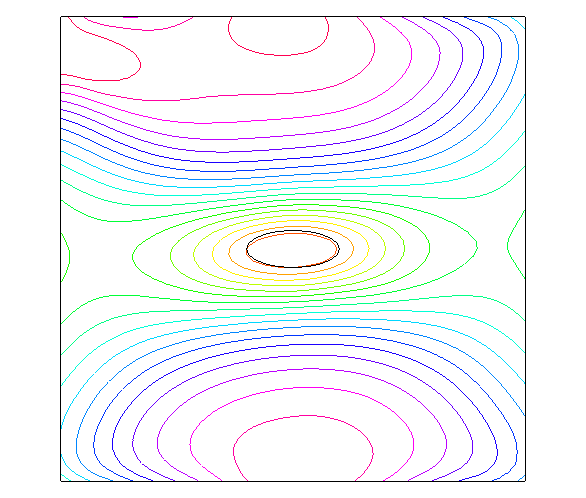}& \includegraphics[width=60mm]{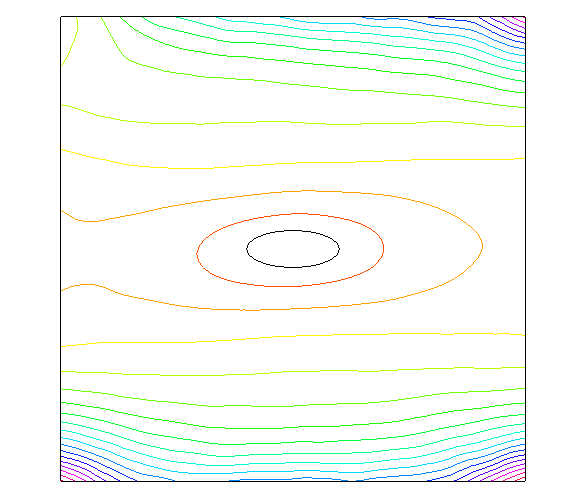}\\
    (c)Iso-values of $\mathcal{G}$ with $5\%$ noise & (d)Iso-values of $\mathcal{G}$ with $10\%$ noise\\
     \includegraphics[width=60mm]{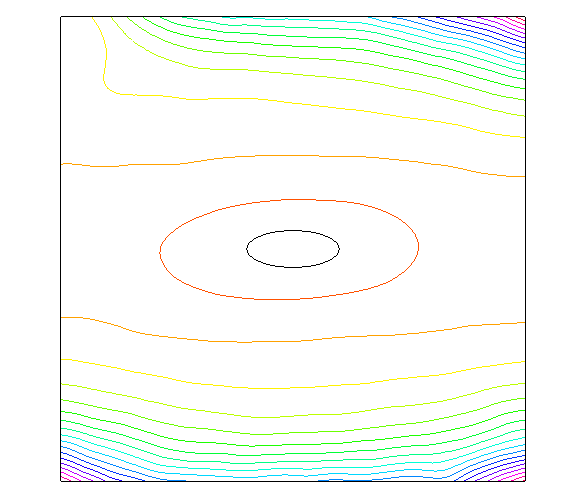}& \includegraphics[width=60mm]{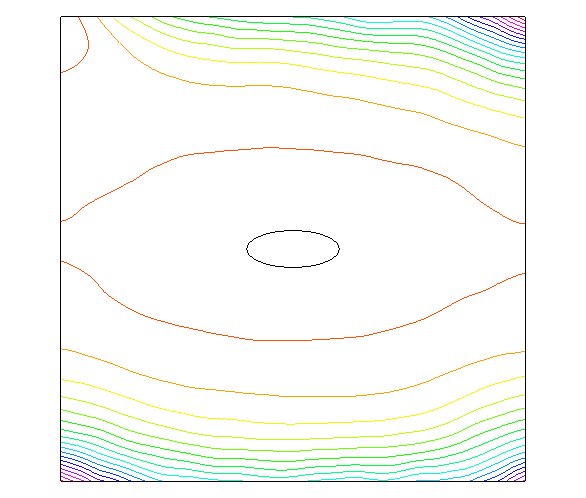}\\
    (e)Iso-values of $\mathcal{G}$ with $20\%$ noise & (f)Iso-values of $\mathcal{G}$ with $30\%$ noise
    \end{tabular}
    \caption{\label{Figure6} Reconstruction with noise data}
    \end{figure}

% ------------------------------------------------------------------------
\end{document}